\documentclass[12pt]{amsart}
\usepackage{amsthm,amssymb,verbatim,amsmath, amsfonts, amscd}
\usepackage{hyperref}

\usepackage{comment}
\theoremstyle{plain}\newtheorem{Theorem}{Theorem}[section]
\theoremstyle{plain}
\theoremstyle{plain}
\theoremstyle{plain}\newtheorem{Lemma}[Theorem]{Lemma}
\theoremstyle{plain}
\theoremstyle{definition}
\theoremstyle{definition}
\theoremstyle{definition}
\theoremstyle{definition}
\theoremstyle{definition}\newtheorem{Question}[Theorem]{Question}
\theoremstyle{definition}
\theoremstyle{definition}
\theoremstyle{definition}
\theoremstyle{definition}
\theoremstyle{definition}
\theoremstyle{definition}
\theoremstyle{definition}
\theoremstyle{definition}\newtheorem{Notation/Definition}
[Theorem]{Notation/Definition}
\theoremstyle{definition}

\def\CF{{\mathcal{F}}}

\def\FM{{\mathfrak{M}}}

\def\Hom{\mathrm{Hom}}

\def\Ind{\mathrm{Ind}}

\def\Park{\mathrm{Park}}

\def\Res{\mathrm{Res}}           

\def\Sc{\mathrm{Sc}}

\newcommand{\SL}{\operatorname{SL}}

\newcommand{\Qd}{\operatorname{Qd}}

\newcommand{\id}{{\text{id}}}

\newcommand{\Z}{{\mathbb{Z}}}

\begin{document}
\title{The Brauer indecomposability of Scott modules
for the quadratic
group Q\lowercase{d}($p$)}
\date{\today}
\author{{Shigeo Koshitani and {\.I}pek Tuvay}}
\address{Center of Frontier Science,
Chiba University, 1-33 Yayoi-cho, Inage-ku, Chiba 263-8522, Japan.}
\email{koshitan@math.s.chiba-u.ac.jp}
\address{Mimar Sinan Fine Arts University, Department of Mathematics, 34380, Bomonti, \c{S}i\c{s}li, Istanbul, Turkey}
\email{ipek.tuvay@msgsu.edu.tr}
\keywords{Scott module, constrained fusion system, the quadratic group}

\maketitle

\begin{abstract}
Let $k$ be an algebraically closed field of prime characteristic $p$ and $P$
a finite $p$-group.
We compute the Scott $kG$-module with vertex $P$ when
$\CF$ is a constrained fusion system on $P$ and $G$ is Park's group
for $\CF$. In the case that $\CF$ is a fusion system of
the quadratic group $\Qd(p)=(\mathbb Z/p\times\mathbb Z/p)\rtimes {\mathrm{SL}}(2,p)$
on a Sylow $p$-subgroup $P$
of $\Qd(p)$ and $G$ is Park's group for $\CF$,
we prove that the Scott $kG$-module with vertex $P$ is Brauer indecomposable.
\end{abstract}

\section{Introduction}
Let $p$ be a prime number and $k$  an algebraically closed field of characteristic
$p$. For a finite group $G$, a $p$-subgroup $Q$ of $G$ and a finite dimensional $kG$-module
$M$, the Brauer quotient $M(Q)$ of $M$ with respect to $Q$ has a natural $kN_G(Q)$-module
structure. A $kG$-module $M$ is said to be {\it Brauer indecomposable} if $M(Q)$ is
indecomposable or zero as a $k (Q\,C_G(Q))$-module for any $p$-subgroup $Q$ of $G$
(see \cite{KKM}).
Brauer indecomposability of $p$-permutation modules is important for verifying categorical
equivalences between $p$-blocks of finite groups (see \cite{IK, KKL, KKM, KL}).

More precisely, if we want to have a stable equivalence of Morita type
between the two principal blocks $B_0(kG)$ and $B_0(kH)$ of finite groups $G$ and $H$,
respectively, such that $G$ and $H$ have a common Sylow $p$-subgroup $P$,
then it is quite natural to expect that the Scott $k(G\times H)$-module
$M=\Sc(G\times H,\,\Delta P)$ with respect to $\Delta P=\{(u,u)\in P\times P|u\in P\}$
would realize the stable equivalence of Morita type.
If this is the case, then for
any non-trivial subgroup $Q$ of $P$
the Brauer quotient $M(\Delta Q)$ induces a Morita equivalence between the principal
blocks $B_0(kC_G(Q))$ and $B_0(kC_H(Q))$ (by the so-called gluing problem due to
Brou\'e, Rouquier and Linckelmann) , and it turns out that $M(\Delta Q)$ has to be
indecomposable as $(kC_G(Q),\,kC_H(Q))$-bimodules because blocks are indecomposable
as bimodules.

A relationship between Brauer indecomposability of $p$-permutation modules and
saturated fusion system is given by R. Kessar, N. Kunugi and N. Mitsuhashi. Namely, in
\cite[Theorem 1.1]{KKM}, they prove that if there
exists a Brauer indecomposable $kG$-module with vertex $P$ then
$\CF_P(G)$ is a saturated fusion system. They also show that the converse of this
statement is not necessarily true. However, when $P$ is abelian and $M$ is the Scott
$kG$-module with vertex $P$ they prove the converse is also true. Later, this
result is extended for saturated fusion systems on $P$ which are not necessarily
abelian under the special case that $P$ is normal in $\CF$ and $G$ is Park's
group for $\CF$ (\cite[Theorems 1.3 and 1.4]{T};
for the construction and details of
Park's group we refer the reader to \cite{P}).
Moreover, in \cite{T} the following question is posed:

\begin{Question}
Let $P$ be a finite group, $\CF$ a saturated fusion system on $P$ and $X$
a characteristic $P-P$-biset for $\CF$. For $G=\Park(\CF, X)$ and $\iota $
Park's embedding of $P$ into $G$,
 is the Scott $kG$-module $\Sc(G, \iota P)$ Brauer indecomposable?
\end{Question}

We answer this question positively in some cases. One of these is the following result
which is the main result of this paper.

\begin{Theorem}\label{main}
Let $P$ be a Sylow $p$-subgroup of
$\Qd(p)=(\mathbb Z/p\times\mathbb Z/p)\rtimes {\mathrm{SL}}(2,p)$
and $\CF=\CF_P(\Qd(p)) $.
Then for
$G={\normalfont\Park}(\CF, \Qd(p))$ and $\iota$ Park's embedding of $P$ into $G$,
the Scott $kG$-module with vertex $\iota P$ is Brauer indecomposable.
\end{Theorem}

As an aside, we prove the following theorem.

\begin{Theorem}\label{side}
Let $\FM$ be a finite group of order $3{\cdot}2^n$ for an integer $n\geq 3$.
Assume that $\FM$ is of $2$-length $2$ and $O_{2'}(\FM)=1$. Let further $P$ be a
Sylow $2$-subgroup of $\FM$, and $\CF=\CF_P(\FM)$. Then for
$G=\Park(\CF, \FM)$ and $\iota$ Park's embedding of $P$ into
$G$, the Scott module $kG$-module with vertex $\iota P$ is Brauer indecomposable.
\end{Theorem}

We shall use the following notation and terminology.
In this paper $G$ is always a finite group, and $k$ is an algebraically closed field of
a prime characteristic $p$.
For the group $\Qd(p)$, see Section 4.
For a subset $\mathcal S$ of $G$ and $g\in G$, we define
$^g \mathcal S=\{gsg^{-1}| s\in \mathcal S\}$ and
$\mathcal S^g=\{g^{-1}sg| s\in \mathcal S\}$.
%Shigeo 6Aug2018
Similarly for $g, x\in G$, we set $^gx=gxg^{-1}$ and $x^g=g^{-1}xg$.  
We write $H\leq G$ if $H$ is a subgroup of $G$ and                     
%%%%%%%%%
$H<G$ if $H$ is a proper subgroup of $G$.
Modules in this paper are finitely generated and left modules.
For a subgroup $H$ and for $kH$- and $kG$-modules $N$ and $M$ respectively,
we write $\Ind_H^G(N) =kG\otimes_{kH}N$ for the induced module of $N$ from $H$ to $G$
(and hence this becomes a $kG$-module),
and $\Res^G_H(M)$ for the restriction of $M$ to $H$ (and hence this becomes a
$kH$-module).
We denote by $k_G=k$ the trivial $kG$-module.
For a subgroup $H$ of $G$ we can define the Scott $kG$-module $\Sc (G,H)$
with respect to $H$, which is an indecomposable direct summand of
$\Ind_H^G(k_H)$ and has $k_G$ in its socle
(see \cite[Chapter 4 \S 8]{NT}).
For a $kG$-module $M$ and a $p$-subgroup $Q$ of $G$ we denote by
$M(Q)$ the Brauer quotient of $M$ with respect to $Q$.
Namely,
$M(Q)=M^Q/\sum_{R<Q}{\mathrm{Tr}}_R^Q(M^R)$
where $M^Q=\{m\in M| \ um=m \text{ for all }u\in Q\}$ and
${\mathrm{Tr}}_R^Q$ is the trace map from $M^R$ to $M^Q$
(see \cite[\S 27]{The}).
For a subgroup $H$ and a normal subgroup $N$ of $G$
we write $G=N\rtimes H$ if $G$ is a semi-direct product of $N$ by $H$.
For a positive integer $n$ we denote by $S_n$ the symmetric group of degree $n$.
For a finite set $\Omega$ we denote by ${\mathrm{Sym}}({\Omega})$ the set of
all bijections of $\Omega$, namely this is isomorphic to $S_{|\Omega|}$
where $|\Omega|$ is the number of the elements of $\Omega$.
For a subgroup $A$ of the symmetric group $S_n$ and a group $H$,
we write $H\wr A$ for the wreath product of $H$ by $A$.
For subgroups $Q, R$ of $G$, let $\Hom_G(Q,R)$ denote the set of all group homomorphisms
from $Q$ to $R$ which are induced by conjugation by some element of $G$. For a $p$-subgroup
$P$ of $G$, the fusion system $\CF_P(G)$ of $G$ on $P$ is the category whose objects are the subgroups
of $P$ and whose morphism set from $Q$ to $R$ is $\Hom_G(Q,R)$ for all $Q,R \leq P$.
When $\mathcal F_P(G)$ is the fusion sytem of $G$ on $P$, $c_g$ for $g\in G$ is defined
in \cite[p.3]{AKO}, and for the notation $O_p(\mathcal F)$ when $\mathcal F$ is a fusion system
see \cite[p.18]{AKO}.  For the other notations in finite group theory
such as $O_p(G)$, $O_{p'}(G)$ and $Z(G)$, see \cite{Gor68};
in representation theory of finite groups see \cite{NT} or \cite{The};
and in fusion systems, see \cite{AKO} or \cite{L}.

The paper is divided into five sections. Section 2 contains results
needed for the proofs of the theorems in later sections.
Section 3 analyzes Park's group and Park's embedding corresponding
to a constrained fusion system. The proof of Theorem \ref{main} is presented
in Section 4 and the proof of Theorem \ref{side} is in Section 5.

\section{Lemmas}

The following lemma is a generalization of Green's Indecomposability Theorem.

\begin{Lemma}[{}{\cite[Lemma 4.2]{IK}}]\label{Green}
Let $G$ be a finite group and $k$ an algebraically closed field
of characteristic $p$. If $H$ is a subgroup of $G$ such that $|G:H|$ is
a power of $p$, then $\Ind_H^G (k)$ is an indecomposable
$kG$-module.
\end{Lemma}

When proving Brauer indecomposability results in the next sections, we are
dealing with Brauer quotients of induced modules several times. The following
lemma makes these computations for the particular case which we are interested in.

\begin{Lemma}\label{BrauerQuotients}
Let $H$ be a subgroup of a finite group $G$ and $P$ a Sylow $p$-subgroup of $H$.
Suppose that $\CF_P(H)=\CF_P(G)$, then for $M=\Ind_H^G(k)$ and for any $Q\leq P$
we have
$$M(Q)= \Ind_{N_H(Q)}^{N_G(Q)}(k) \mbox{ and } \Res_{Q\,C_G(Q)}^{N_G(Q)} (M(Q))=\Ind_{QC_H(Q)}^{Q\,C_G(Q)}(k). $$

\end{Lemma}

\begin{proof}
From \cite[1.4]{B}, we have $M(Q)= \oplus_{g} \Ind_{N_G(Q)\cap\, {^g\!H}}^{N_G(Q)} (k)$
where $g$ runs over representatives of double cosets $N_G(Q) \backslash T_G(Q,H) / H$
for
\linebreak
$T_G(Q,H)=\{ g\in G \ | \  Q^g \leq H \}$. We claim that $T_G(Q,H)= N_G(Q) H$.
Indeed, if $g \in T_G(Q,H) $ the condition
$Q^g \leq H$ implies that there exists $x\in H$ such that $Q^{gx} \leq P$ since $P$ is
a Sylow $p$-subgroup of $H$. Thus $c_{gx}:Q \to Q^{gx}$ is in $\CF_P(G)=\CF_P(H)$ which
implies $gx$ to be in $ C_G(Q) H$. Hence $g \in  C_G(Q) H$ which shows the containment
$T_G(Q,H) \leq  N_G(Q) H$ and the other direction is trivial. Therefore, there exists only
one coset above and the first statement of the lemma is established.

Since as a $k N_G(Q)$-module, $M(Q)$ is equal to $\Ind_{N_H(Q)}^{N_G(Q)}(k)$, by the Mackey formula
$$\Res_{Q\,C_G(Q)}^{N_G(Q)} (M(Q))=\bigoplus_{g} \Ind_{Q\,C_G(Q) \cap\, {}^g\!N_H(Q)}^{Q\,C_G(Q)}(k)$$
where $g$ runs over a set of representatives of double cosets \linebreak$ Q\,C_G(Q)\backslash N_G(Q) / N_H(Q)$
in $N_G(Q)$. By the equality of the fusion systems, we have $C_G(Q) N_H(Q)= N_G(Q)$, so there is only
one coset above. Therefore, $\Res_{Q\,C_G(Q)}^{N_G(Q)} (M(Q))=\Ind_{Q\,C_H(Q)}^{Q\,C_G(Q)}(k). $
\end{proof}

The following lemma will be used in the proofs of Theorems \ref{main} and \ref{side}.
Note that this lemma can also be seen as a generalization of Green's Indecomposability Theorem.

\begin{Lemma}\label{Scott-permutation}
Let $H$ be a subgroup of a finite group $G$ and $P$ a Sylow $p$-subgroup of $H$
where $\CF_P(H)=\CF_P(G)$. Assume that there exists a normal $p$-subgroup $N$ of $G$
such that $C_G(N \cap H)$ is a $p$-group, then we have
$$\Sc(G,P)= \Ind_H^G(k).$$
\end{Lemma}

\begin{proof}
Since $P$ is a Sylow $p$-subgroup of $ H$, the $kG$-Scott module $\Sc(G, P)$
is a direct summand of $\Ind_H^G(k)$
(see \cite[Chapter 4, Corollary 8.5]{NT}). Thus, there is a $kG$-module
$X$ such that
$$\Ind_H^G (k) = \Sc(G, P) \oplus X.$$

Set $R=N \cap H$, then $R$ is a normal subgroup of $P$. If $R$ is
trivial, the condition on the centralizer of $R$ in $G$ implies
that $G$ is a $p$-group, so $\Ind_H^G (k)$ is indecomposable by Green's
Indecomposability Theorem and hence $\Ind_H^G (k) = \Sc(G, P)$ in this
case. Assume $R$ is non-trivial and set $I=\Ind_H^G (k)$ and $S=\Sc(G,P)$.
By Lemma \ref{BrauerQuotients} $I( R)=\Ind_{N_H( R)}^{N_G(R)} (k) $. Moreover,
by the equality of fusion systems, we have
$$|N_G(R)|/|C_G(R)| = |N_H(R)| / |C_H(R)|.$$
So the quotient $|N_G(R):N_H(R)|$ is a power of $p$ since $C_G(R)$ is a p-group.
Thus Lemma \ref{Green} implies that
$I(R)=\Ind_{N_H(R)}^{N_G(R)} (k)$ is
indecomposable. Therefore, the identity
$$I(R)= S(R) \oplus X(R)$$
implies that either $X(R)=0$ or $S(R)=0$.
If the latter case occurs, since $R$ is normal in $P$,
as $N_G(P) \cap N_G(R)$-modules
$$0 \neq S( P)\cong S(R)( P)=0 $$
(see \cite[Proposition 1.5(3)]{BP})
where the left hand side is non-zero since $S$ has vertex $P$. Hence
we get a contradiction. Thus, we have $X(R)=0$.

On the other hand, since $X \ | \ \Ind_{H}^G (k)$ by the Mackey formula
$$\Res_N^G (X) \ \Big| \ \bigoplus_g \Ind_{N\cap \, ^g\!H}^N (k)$$
where $g$ runs through representatives of double cosets in
$N \backslash G / H$. By Green's Indecomposability Theorem each
summand above is indecomposable since
$N$ is a $p$-group. So, if $X$ is non-zero, $\Res_N^G (X)$ is isomorphic 
to a direct sum of some of the permutation modules $\Ind_{N\cap \, ^g\!H}^N (k).$
Moreover, $N$ being normal in $G$ implies $N\cap \, ^g\!H = \ ^g (N\cap  H ) =\ ^g\!R$.
On the other hand, by \cite[1.4]{B} $(\Ind_{N\cap \, ^g\!H}^N (k))( ^g\!R)$ is non-zero. Hence, 
if $X$ is non-zero, there exists a direct summand $X'$ of $X$ such that $X'(^g \!R) \neq 0$ .
This contradicts with $X(R)$ being zero. Therefore $X$ must be zero.
\end{proof}

\section{Park's construction in the case of constrained fusion systems}

Recall that a saturated fusion system $\CF$ on $P$ is called {\it constrained} if
$C_P(O_p(\CF)) \leq O_p(\CF)$. This definition is analogous with the definition
of $p$-constrained groups. Recall that a finite group $G$ is called
{\it $p$-constrained} if $O_p(G/ O_{p'}(G))$
contains its own centralizer in $G/O_{p'}(G)$ and called {\it strictly $p$-constrained}
if $G$ is $p$-constrained and $O_{p'}(G)=1.$  It is shown in \cite[Proposition 4.3]{BCGLO}
that every constrained fusion system $\CF$ on $P$ is realized by a unique strictly $p$-constrained
group containing $P$ as a Sylow $p$-subgroup, which is called a {\it model for $\CF$}.

Let $\CF$ be a constrained fusion system on $P$ and $\FM$ be a
model for $\CF$. Since $P$ is a Sylow subgroup of $\FM$ and $\CF=\CF_P(\FM)$,
the group $\FM$ can be regarded as a characteristic $P-P$-biset for $\CF$
(for definition and existence result of a characteristic biset for a saturated
fusion system, we refer the reader to \cite[Proposition 5.5]{BLO}).
Thus $\FM$ can be used to construct
Park's group for $\CF$. Recall that $\Park(\CF, \FM)$ is the
group of bijections of the $P-P$-biset $\FM$ preserving the right
$P$-action. Letting $G=\Park(\CF, \FM)$, Park's embedding
$\iota: P \hookrightarrow G$ is defined as $\iota: u\mapsto (m\mapsto u m)$ for all 
$m \in M$ so that 
$\CF \cong \CF_{\iota P}(G)$ ( \cite[Theorem 3]{P}).
As mentioned in \cite[Line 12 on page 408]{P}, $G$ is isomorphic to the
wreath product $P \wr S_n$ where $n$ is the index $|\FM:P|$. Indeed, this 
isomorphism can be given as follows: letting $\FM = \bigsqcup_{i=1}^n m_i P$, and observing that 
each $f \in G$ is determined by its value on the coset representatives $m_i$ for $i=1, \ldots , n$, 
$f \mapsto (x_1, \ldots, x_n; \sigma_f)$ where $f(m_i)=m_{\sigma_f(i)} x_{\sigma_f(i)}.$

Throughout the rest of the paper, let us
identify $G$ with $P \wr S_n$ whenever $G=\Park(\CF, \FM)$ where $\FM$ is a
model for the constrained fusion system $\CF$ and $n$ is the index of $P$ in $\FM$.

\begin{Lemma}\label{P}
Let $\CF$ be a constrained fusion system on $P$ and $\FM$ a model for
$\CF$. Let also $G=\Park(\CF, \FM)$ then
$$\iota P=\{(m_{1}^{-1} u m_{\sigma_u^{-1}(1)}, \ldots, m_n^{-1} u m_{\sigma_u^{-1}(n)}; \sigma_u) \ | \ u \in P \}$$
where $\{m_1, \ldots, m_n\}$ is a set of left coset representatives of $P$ in $\FM$ and $\sigma_u=\sigma_{\iota(u)}$ 
for all $u \in P$.
\end{Lemma}

\begin{proof}
Let us write $\FM = \bigsqcup_{i=1}^n m_i P$ as a disjoint union of left cosets of $P$.
Since any automorphism in $G$ is determined by its value on the
coset representatives $m_i$, we can compute the subgroup $\iota P$ as follows. Let
$u$ be an arbitrary element of $P$, then the image of $\iota(u)$ in
$G$ is $(x_1, \ldots, x_n; \sigma_u)$ where
$$\iota (u)(m_i)=um_i=m_{\sigma_u(i)} x_{\sigma_u(i)}$$
which is a direct consequence of the isomorphism between $G$ and $P \wr S_n$ 
as explained before the statement of the lemma. 
\end{proof}

Since a model $\FM$ acts on itself by left multiplication so that this action gives an automorphism of 
$\FM$ preserving the right $P$-action, the embedding $\iota: P \to G$ 
extends to $\FM$. We will denote this extended map also by $\iota$.

\begin{Lemma}\label{Op(F)}
Let $\CF$ be a constrained fusion system on $P$ and $\FM$ a model for
$\CF$. Let $G$ be equal to $\Park(\CF, \FM)$ and $B$ the base group of $G$. Then 
$\iota O_p(\CF) = B \cap (\iota P)$ 
and as a subgroup of $G$
$$\iota O_p(\CF)=\{ (u ^{m_1}, \ldots, u^{m_n}; {\normalfont\id}) \ | \ u \in O_p(\CF) \}$$
where $\{m_1, \ldots, m_n\}$ is a set of left coset representatives of $P$ in $\FM$.
\end{Lemma}

\begin{proof}
Since $O_p(\CF)$ is normal in $\FM$, any $u \in O_p(\CF)$ satisfies $m_i^{-1} u m_i \in O_p(\CF)$ so $u m_i P = m_i P$.
Hence, for all $u \in O_p(\CF)$ we have $\sigma_u= \id$. This implies
$$\iota (u) = ( m_1^{-1} u m_1, \ldots, m_n^{-1} u m_n; \id)$$ which shows
why the second statement is true.

Set $R=B \cap (\iota P)$, then we have $R=B \cap \iota (\FM)$, 
so $R$ is a normal subgroup of $\iota \FM$ and $\iota P$. Hence $R$ is fully $\CF$-normalized
and by \cite[Proposition 3.16]{L}
$$ N_{\CF}(R) \cong \CF_{N_{\iota P}(R)} (N_{\iota \FM}(R)) $$
which yields $ N_{\CF}(R)=\CF$.
So, $R$ is normal in $\CF$ and this implies $R \leq \iota O_p(\CF)$. By the
first paragraph, we have $\iota O_p(\CF) \leq B$, hence $\iota O_p(\CF) \leq R$.
\end{proof}

\begin{Lemma}\label{Op(F)2}
Let $\CF$ be a constrained fusion system on $P$ and $\FM$ a model for
$\CF$. Set $G=\Park(\CF, \FM)$. Then
$C_G(\iota O_p(\CF))$ is a $p$-group.
\end{Lemma}

\begin{proof}
Let $(x_1, \ldots, x_n;\beta) \in C_G(\iota O_p(\CF))$
for $x_i\in P$ and $\beta\in S_n$. Then we have by Lemma \ref{Op(F)} 
that for all $u \in O_p(\CF)$
$$(x_1, \ldots, x_n;\beta) ( u^{m_1}, \ldots, u^{m_n}; \id)
(x_1, \ldots, x_n; \beta)^{-1} 
=(u^{m_1}, \ldots, u^{m_n}; \id),$$ 
where $m_1, \cdots, m_n$ are the same as in Lemma \ref{Op(F)}, namely
$$(x_1, \ldots, x_n;\beta) ( u^{m_1}, \ldots, u^{m_n}; \id)
(x_{\beta (1)}^{-1}, \ldots, x_{\beta(n)}^{-1};\beta^{-1})=(u^{m_1}, \ldots, u^{m_n}; \id).$$ 

This implies that
for all $i=1, \ldots, n$  and for all $u\in O_p(\CF)$
$$ x_i m^{-1}_{\beta^{-1}(i)} u m_{\beta^{-1}(i)}x_i^{-1} = m_i^{-1} u m_i$$
which is equivalent to $m_i x_i m_{\beta^{-1}(i)}^{-1}$ to be in $C_{\FM}(O_p(\CF))$
for all $i=1, \ldots, n$. Since $\FM$ is
$p$-constrained and $O_p(\CF)$ is $\CF$-centric \cite[Proposition 4.12]{L} implies $ m_i x_i m_{\beta^{-1}(i)}^{-1}\in Z(O_p(\CF))$.
So for all $i$, there exists $z_i \in Z(O_p(\CF))$ such that $ m_i x_i m_{\beta^{-1}(i)}^{-1}= z_i$ or
equivalently $$ m_i x_i = z_i m_{\beta^{-1}(i)}.$$ Since $Z(O_p(\CF))$ is a characteristic
subgroup of $O_p(\CF)$ and $O_p(\CF)$ is normal in $\FM$, $Z(O_p(\CF))$ is normal in $\FM$. Thus
for any $i$, there exists $z_i' \in Z(O_p(\CF))$ with
$z_i m_{\beta^{-1}(i)}=m_{\beta^{-1}(i)} z_i'$. This implies that $$ m_i x_i=m_{\beta^{-1}(i)} z_i',$$
so $m_i P= m_{\beta^{-1}(i)} P$ for all $i$. Hence $\beta =\id$ and $C_G(\iota O_p(\CF)) \leq B.$
\end{proof}

\begin{Theorem}\label{scott-constrained}
Let $\CF$ be a constrained fusion system on $P$, $\FM$
a model for $\CF$ and $G=\Park(\CF, \FM)$ and $\iota$
Park's embedding of $\FM$ into $G$.
Then
$\Sc(G, \iota P) = \Ind_{\iota \FM}^G(k)$.
\end{Theorem}

\begin{proof}
By Lemma \ref{Op(F)2}, $C_G(\iota O_p(\CF))$
is a $p$-group. Also $B$ is normal in $G$ and satisfies $B \cap \iota \FM=\iota O_p(\CF)$
by Lemma \ref{Op(F)}.
The result follows from Lemma \ref{Scott-permutation}.

\end{proof}

\section{The quadratic group case}
Consider the semidirect product of an elementary abelian group
$V$ of order $p^2$ by $\SL(2,p)$ where $\SL(2,p)$ acts naturally on $V$.
Glauberman called 
it the {\it quadratic group} and denoted it
by $\Qd(p)$ (see \cite[p.1104]{Gla}). Actually this group plays a very
important role in finite group theory and
representation theory of finite groups (see \cite[pp.81 and 95]{L}).

Set $\FM=\Qd(p)$ and a Sylow $p$-subgroup $P$ of $\FM$ as
$P=V\rtimes \langle \alpha \rangle $ where
$$\alpha=\left( {\begin{array}{cc}
   1 & 1 \\
   0 & 1 \\
  \end{array} } \right).$$ Note that $P$ is isomorphic to the extraspecial group of order $p^3$ with
  exponent $p$. We have
  $$N_{\FM}(P)=V\rtimes \big{\{} \left( {\begin{array}{cc}
   r & l \\  0 & r^{-1} \\ \end{array} } \right) \ | \ r \in (\Z/p) ^*,\ l \in \Z/p \big{\}}. $$
Set $M_1=N_{\FM}(P)$, then there are $p-1$ left cosets of $P$ in $M_1$, and they explicitly are
$$\{ \nu_{r}  P \ | \ r= 1, \ldots, p-1 \},$$
  where $\nu_{r}=\left( {\begin{array}{cc} r & 0 \\ 0 & r^{-1} \\
  \end{array} } \right)$
  and the left cosets of $M_1$ in $\FM$ are
$$\{M_1, \beta M_1, \ldots, \beta^{p-1} M_1, \gamma M_1 \}$$ where
  $\beta=\left( {\begin{array}{cc}
   1 & 0 \\
   1 & 1 \\
  \end{array} } \right)$ and $\gamma= \left( {\begin{array}{cc}
   0 & -1 \\
   1 & 0 \\
  \end{array} } \right)$.
Let us gather this information and choose representatives of left cosets of $P$ in $\FM$ as follows:
\[m_j= \begin{cases}
      \beta^s   \nu_{r+1} & \text{ if } 1 \leq j \leq p^2-p \\
      \gamma \nu_{r+1} & \text{ if } p^2-p < j \leq p^2-1 
         \end{cases}
\]
where $j-1= s(p-1) + r$ for $1 \leq j \leq p^2-p$ and $j-1=p(p-1)+r$ for $p^2-p<j\leq p^2-1$
with $0\leq s \leq p-1$ and $0 \leq r<p-1$  . 
Thus we get a set of left cosets $\{m_j P\ | \ j=1, \ldots, p^2-1\}$
of $P$ in $\FM$.

Set $\CF=\CF_P(\FM)$ and $G=\Park(\CF, \FM)$. Then $G= P\wr S_n$ for $n=p^2-1$ and $O_p(\CF)=V$.
The proof of the following lemma is obvious since $P$ is normal in $M_1$ and $M_1$ is the union of the
first $p-1$ cosets. We keep the same notation as in Section 3.

\begin{Lemma}\label{Fix1&p-1}
For $u\in P$, the permutation $\sigma_u$ fixes all elements in $\{1,\ldots,p-1\}$.
\end{Lemma}

\begin{Lemma}\label{p-1-pcycles}
Let $p$ be an odd prime. For $u\in P - V$, the permutation $\sigma_u$ is a product of $p-1$ disjoint $p$-cycles in
${\mathrm{Sym}}({\{p, \ldots, n\}})$ where $n=p^2-1$.
\end{Lemma}

\begin{proof}
For $u\in P - V$, we have $u = v \alpha^i$ for some $v\in V$ and $i\in \{1, \ldots, p-1\}$.
So $\sigma_u=\sigma_{\alpha^i}$ by Lemma \ref{Op(F)}.
Hence, without loss of generality we can choose $u=\alpha$.
From Lemma \ref{Fix1&p-1}, $\sigma_u(i)=i$ for $i=1, \ldots, p-1$. Moreover, we have
$$\alpha \beta^s = \beta^{s(s+1)^{-1}}  \nu_{s+1} \ \alpha^{(s+1)^{-1}} $$
if $s+1$ is invertible in $\Z/p$ and
$$\alpha \beta^{p-1}= \gamma  \nu_{p-1}  \alpha^{p-1}. $$
Also, $ \alpha \gamma =\beta  \alpha^{p-1}.$ 
Therefore, $u (\beta^s  P) =\beta^{s(s+1)^{-1}} \nu_{s+1} P$ if $s+1$ is
invertible in $\Z/p$ and $u (\beta^{p-1} P) = \gamma  \nu_{p-1} P$. Also,
$ u (\gamma P)=\beta P$. As a result, since $\nu_{r}$ 
normalizes $P$, we have 

\[u \ (\beta^s  \nu_{r} P)= \begin{cases}
      \beta^{s(s+1)^{-1}} \nu_{s+1}  \nu_{r} P & \text{ if $s+1$ is invertible in $\Z/p$} \\
    \gamma  \nu_{p-1}  \nu_{r} P & \text{ if $s=p-1$}
         \end{cases}
\]
and $ u (\gamma  \nu_{r} P)=\beta \nu_{r} P$ for all $0\leq s \leq p-1$ and $1 \leq r \leq p-1.$

At this point some combinatorics is needed. Let $f: (\Z/p)^*\backslash \{p-1\} \to (\Z/p)^* \backslash \{1\}$ which is defined as 
$f(s)=s(s+1)^{-1}$. Then $f$ is a bijection between these sets. Hence 
$$ u \ (\beta  \nu_{r} P) =  \beta^{f(1)} \nu_{2}  \nu_{r} P$$

$$u \ (\beta^{f(1)} \nu_{2}  \nu_{r} P) = \beta^{(f\circ f )(1)} \nu_{f(1)+1} \nu_{2}  \nu_{r} P$$
and in general 
$$u \ (\beta^{f^{(i)}(1)} \nu_{g(i-1)}  \nu_{r} P) = \beta^{f^{(i+1)} (1)} \nu_{g(i)}  \nu_{r} P$$
where $f^{(k)}$ denotes composition of $f$ with itself $k$ times (with the convention that $f^{(0)}(1)=1$) 
and $g(i)=\prod_{k=0}^{i}(f^{(k)}(1)+1).$ Note that this process continues until $(f^{(i+1)}(1)+1)$ is not 
invertible in $\Z/p$ which is equivalent to say that $f^{(i+1)}(1)=p-1$ in $\Z/p$. Note also that we have 
$$f(1)=2^{-1}, \ f^{(2)}(1)= 2^{-1} (2^{-1}+1)^{-1}= (2(2^{-1}+1))^{-1}= 3^{-1}$$
so by induction it is easy to see that $f^{(s)}(1)=(s+1)^{-1}$ whenever $(s+1)$ is invertible in $\Z/p$. 
From this, we deduce that $s=p-2$ is the least positive integer satisfying $f^{(s)}(1)=p-1$ in $\Z/p$.

After applying $u$ $p-2$ number of times to $m_{p+r-1} P=\beta \nu_{r} P,$ we reach the coset 
$(\beta^{p-1} \nu_{g(p-3)} \nu_{r} P).$ Since $p-1$ is the only element in $(\Z/p)^*\backslash \{1\}$ 
which is equal to its inverse, we have 
$g(p-3)=\prod_{k=2}^{p-1}k=p-1.$ So $\nu_{g(p-3)}=\nu_{p-1} $ 
and the coset we reach is $(\beta^{p-1}  \nu_{p-1}\ \nu_{r} P).$ If we apply $u$ to this coset 
we reach $\gamma  (\nu_{p-1})^2 \ \nu_{r} P=\gamma  \nu_{r} P$. Finally, if we apply $u$ to 
this coset we get $(\beta  \nu_{r} P)$. Therefore, we reach the coset that we started in $p$ 
number of steps, which corresponds to a $p$-cycle in ${\mathrm{Sym}}({\{p, \ldots, n\}})$, let us denote this 
cycle by $\sigma_{r}$. 
The identity 
 $$m_{(\sigma_{r})^i(p+r-1)} P =  \alpha^i(\beta \nu_{r} P)$$
gives the $i+1$st position of $\sigma_{r}$'s presentation as above for
$i=0, 1,\ldots, p-1$ and $1\leq r \leq p-1.$ More precisely, the presentation of the coset at the right hand side in terms of $m_j$'s  
gives the information about $\sigma_r$. Moreover, as can be seen from the computations above, as $r$ changes, we get $p$-cycles 
disjoint from each other. As a result, we get $\sigma_u= \prod_{r=1}^{p-1} \sigma_{r}.$ 

\end{proof}

\begin{Lemma}\label{pgroup}
Let $p$ be an odd prime. If $Q$ is a subgroup of $P$ of order $p^2$,  then $C_G(\iota Q)$ is a $p$-group.
\end{Lemma}

\begin{proof}
There are two types of subgroups of order $p^2$ in $P$.
If $Q=V$, from Lemma \ref{Op(F)2} $C_G(\iota Q)$ is a $p$-group.
If $Q\neq V$ then $Q= \langle t \rangle \times \langle y \rangle$
where $t=(1,0)$ is a central element of $P$ and $y=(a,b) \alpha$
for some $a,b \in \Z/p$. We have $C_G(\iota Q)= C_G(\iota t) \cap C_G(\iota y)$.
Since $t\in O_p(\CF)$, by Lemma \ref{Op(F)} $\iota t=(t^{m_1}, \ldots, t^{m_n}; \id)$.
So we have
$$C_G(\iota t)=\{ (x_1, \ldots, x_n; \tau) \ | \ m_{\tau^{-1}(i)} x_i^{-1} m_i^{-1} \in C_M(t) \text{ for all } i=1, \ldots, n\} $$
and since $ C_M(t)=P$, this becomes
$$C_G(\iota t)=\{ (x_1, \ldots, x_n; \tau) \ | \ m_{\tau^{-1}(i)} x_i^{-1} m_i^{-1} \in P\text{ for all } i=1, \ldots, n \}. $$
In particular, the condition for $(x_1, \ldots, x_n; \tau)$
to be in $C_G(\iota t)$ implies that
$m_{\tau^{-1}(i)} x_i^{-1} m_i^{-1} \in P$ for all $i=1,\ldots, p-1$. This is equivalent to say that 
for all  $i=1,\ldots, p-1$, we have $m_{\tau^{-1}(i)} P = P m_i$ which
is equal to $m_i P$ since $P$ is normal in $M_1$. Thus $\tau(i)=i$ for $i=1, \ldots, p-1$.
From Lemmas \ref{P} and \ref{p-1-pcycles},
$$\iota y=(m_1^{-1}ym_{\sigma_y^{-1}(1)}, \ldots, m_n^{-1}ym_{\sigma_y^{-1}(n)}; \sigma_y)$$ where 
$\sigma_y=\prod_{r=1}^{p-1} \sigma_r $ where for $r=1, \ldots, p-1$, the permutations $\sigma_r$ are
$p$-cycles disjoint from each other. So if an element
$(x_1, \ldots, x_n; \tau)$ centralizes $\iota y$, then $\tau$ centralizes the permutation
$\sigma_y$.
By \cite[4.1.19]{JK}, we have $C_{{\mathrm{Sym}}({\{p, \ldots, n\}})}(\sigma_y) \cong (\Z/p) \wr S_{p-1}$.
Here the wreathing part comes from the permutations which permute the $p-1$ cycles lying
in $\sigma_y$. Since $\tau(i)=i$ for $i=1, \ldots, p-1$, we deduce that if
$\allowbreak (x_1, \ldots, x_n; \tau) \in C_G(\iota Q)$
then $\tau \in C_{{\mathrm{Sym}}({\{p, \ldots, n\}})}(\sigma_y)$.

Suppose to the contrary that $C_G(\iota Q)$ is not a $p$-group. Then by Cauchy's Theorem
there exists a prime $q$ different from $p$ and an element $(x_1, \ldots, x_n; \tau)$
in $C_G(\iota Q)$ of order $q$. In this case, $\tau$ has order $q$ and can be seen as an element of 
$S_{p-1}$ by the remarks in the previous paragraph. Hence, for all $r\in \{1, \ldots, p-1\}$ we have 
$\tau \sigma_r \tau^{-1}=\sigma_{\tau(r)}$. Since $\tau$ is non-identity there exist $r, r' \in \{1, \ldots, p-1\}$ 
with $r \neq r'$ such that $r'= \tau(r)$. So the cycles $\sigma_r$ and $\sigma_{r'}$ are disjoint. In this case, 
there exists an integer $j$
appearing in the cycle $\sigma_{r'}$ such that $\tau(p+r-1)=j$.
Since $j$ appears in $\sigma_{r'}$, we have $m_j P= \alpha^{j'} \beta \nu_{r'} P$ for some $j' \in \{0, 1, \ldots, p-1\}$. 
The condition for $(x_1, \ldots, x_n; \tau)$ to be
in $C_G(\iota t)$ implies in this case that $m_{p+r-1} x_j^{-1} m_j^{-1}$ to be in $P$.
Letting $x_j=v \alpha^{i}$ for some $v\in V$ and $i \in \Z/p$,
and observing that the wreathing part of $m_{p+r-1} x_j^{-1} m_j^{-1}$  is 
$\beta \nu_{r} \alpha^{-i} \nu_{(r')^{-1}} \beta^{-1}$, we should have 
$\beta \nu_{r} \alpha^{-i} \nu_{(r')^{-1}} \beta^{-1}$,
to be in $\langle \alpha \rangle$.
But the matrix
$$\left( {\begin{array}{cc} r(r')^{-1}+i r r' & -i r r' \\  r(r')^{-1}+i r r'- r^{-1}r' & -irr'+r^{-1}r' \\ \end{array} } \right)$$
lies in $\langle \alpha \rangle$ if and only if $r=r'$ and $i=0$. This
contradicts with $\sigma_r$ and $\sigma_{r'}$ to be disjoint.
Therefore $C_G(\iota Q)$ is a $p$-group in this case, too.

\end{proof}

\begin{proof}[Proof of Theorem \ref{main}]

First, let us consider the case where $p=2$. 
Then it is easy to know that $\Qd(2)\cong S_4$, so that
Theorem \ref{main} is covered by Theorem \ref{side}.

Now, let us assume that $p$ is an odd prime.
Set $\FM=\Qd(p)$. Since $\FM$ is $p$-constrained, $\CF=\CF_P(\FM)$ is a constrained fusion system, so Theorem
\ref{scott-constrained} implies that $\Sc(G, \iota P)=\Ind_{\iota \FM}^G( k)$. Moreover
since $\CF_{\iota P}(\iota \FM) = \CF_{\iota P}(G)$, Lemma \ref{BrauerQuotients}
implies that
$$\Res_{\iota Q\,C_G(\iota Q)}^{N_G(\iota Q)} ((\Sc(G, \iota P))(\iota Q))=\Ind_{\iota Q\,C_{\iota \FM}(\iota Q)}^{\iota Q\,C_G(\iota Q)}(k)$$
for all $Q \leq P$. Thus, using \cite[Theorem 1.3]{IK} it is enough to show that
$\Ind_{\iota Q\,C_{\iota \FM}(\iota Q)}^{\iota Q\,C_G(\iota Q)}(k) $ is
indecomposable for all fully $\CF$-normalized $Q\leq P$.

If $Q\leq P$ such that $|Q| \geq p^2$ then by Lemma \ref{pgroup} $C_G(\iota Q)$ is a
$p$-group. Hence Green's Indecomposability Theorem implies that
$\Ind_{\iota Q\,C_{\iota \FM}(\iota Q)}^{\iota Q\,C_G(\iota Q)}(k) $ is
indecomposable in this case.

Since $Q$ is fully $\CF$-normalized, \cite[Proposition 2.5]{L} implies $Q$
is fully $\CF$-centralized. Hence by \cite[Lemma 2.10]{L}
$\iota Q\,C_{\iota P}(\iota Q)$ is a Sylow $p$-subgroup of $\iota Q\,C_{\iota \FM}(\iota Q)$.
Moreover $\CF_{\iota P}(\iota \FM)= \CF_{\iota P}(G)$ implies
$$\CF_{\iota Q\,C_{\iota P}(\iota Q)}(\iota Q\,C_{\iota \FM}(\iota Q))
=\CF_{\iota Q\,C_{\iota P}(\iota Q)}(\iota Q\,C_G(\iota Q)).$$

If $Q \leq P$ of order $p$, there are two possibilities: either $Q < V$ or
$Q \nless V$. If $Q < V$, then $Q$ is $\CF$-conjugate to
$Z(P)$. Since $Z(P)$ is the only fully $\CF$-normalized subgroup of $P$ 
with the property that $Q < V$, 
it follows that 
$Q=Z(P)$. Then $Q\,C_V(Q)=V$.
So we have
$$C_{\iota Q\,C_G(\iota Q)}(\iota (Q\,C_V(Q)))=C_{\iota Q\,C_G(\iota Q)}(\iota V) \le
C_G(\iota V).$$
Lemma \ref{Op(F)2} implies that $C_G(\iota V)$ is a $p$-group so
$C_{\iota Q\,C_G(\iota Q)}(\iota (Q\,C_V(Q)))$ is also a $p$-group. Hence we can apply Lemma
\ref{Scott-permutation} by changing $G$ with $\iota Q\,C_G(\iota Q)$, $H$ with
$\iota Q\,C_{\iota M}(\iota Q)$, $P$ with $\iota Q\,C_{\iota P}(\iota Q)$ and $N$
with $\iota Q\,C_B(\iota Q)$ and deduce that
$$\Ind_{\iota Q\,C_{\iota \FM}(\iota Q)}^{\iota Q\,C_G(\iota Q)}(k)
=\Sc(\iota Q\,C_G(\iota Q), \, \iota Q\,C_{\iota P}(\iota Q)).$$
If $Q \nless V$ which is fully $\CF$-normalized, then $ Q\,C_V(Q)=Z(P) \times Q$ is a subgroup of $P$ of order $p^2$.
So $C_G(\iota Q\,C_{\iota V}(\iota Q))$ is a $p$-group by Lemma
\ref{pgroup} which implies that $C_{\iota Q\,C_G(\iota Q)}(\iota Q\,C_{\iota V}(\iota Q))$
is also a $p$-group. Hence we can apply Lemma \ref{Scott-permutation}
by changing $G$ with $\iota Q\,C_G(\iota Q)$, $H$ with
$\iota Q\,C_{\iota \FM}(\iota Q)$, $P$ with $\iota Q\,C_{\iota P}(\iota Q)$ and $N$
with $\iota Q\,C_B(\iota Q)$ and deduce
that
$$\Ind_{\iota Q\,C_{\iota \FM}(\iota Q)}^{\iota Q\,C_G(\iota Q)}(k)
=\Sc(\iota Q\,C_G(\iota Q), \iota Q\,C_{\iota P}(\iota Q))$$
in this case, too. Therefore, we deduce that $\Ind_{\iota Q\,C_{\iota \FM}(\iota Q)}^{\iota Q\,C_G(\iota Q)}(k)$
is indecomposable for all possible fully $\CF$-normalized subgroups $Q$ of $P$ as desired.
\end{proof}

\section{The $3{\cdot}2^n$ ordered group case}

\begin{proof}[Proof of Theorem \ref{side}]
Since $G = P \wr S_3$, the index $|G: \iota \FM|$ is a power of $2$. Hence
Lemma \ref{Green} together with \cite[Chapter 4, Corollary 8.5]{NT} implies that
$$\Sc(G, \iota P) = \Ind_{\iota \FM}^G(k).$$
Since $\CF_{\iota P}(\iota \FM)=\CF_{\iota P}(G)$, Lemma \ref{BrauerQuotients}
implies that
$$\Res_{\iota Q\,C_G(\iota Q)}^{N_G(\iota Q)} ((\Ind_{\iota \FM}^G (k)) (\iota Q)) =\Ind_{\iota Q\,C_{\iota \FM}(\iota Q)}^{\iota Q\,C_G(\iota Q)}(k)$$
for all $Q \leq P$.
So it is enough to show the indecomposability of
$\Ind_{\iota Q\,C_{\iota \FM}(\iota Q)}^{\iota Q\,C_G(\iota Q)}(k)$ for all
$Q \leq P$.

$\FM$ is solvable and $O_{2'}(\FM)=1$, so $\FM$ is strictly $2$-constrained and
$\CF=\CF_P(\FM)$ is a constrained fusion system. Since $\FM$ has $2$-length $2$,
the non-trivial subgroup $O_2(\FM)$ is strictly contained in $P$. Note also
that $O_2(\FM)=O_2(\CF)$.

If $ Q \not \leq O_2(\CF)$, then there exists an element $u \in Q$ such that
$\sigma_u$ is a non-trivial $2$-cycle by Lemma \ref{Op(F)}.
Thus, if $(x_1, x_2, x_3; \beta) \in C_G(\iota u)$, then $\beta \in C_{S_3}(\sigma_u)= \langle \sigma_u \rangle$.
Hence any element in $C_G(\iota u)$ is an element of order a power of $2$ and $C_G(\iota Q)$
is a $2$-group. Therefore $\Ind_{\iota Q\,C_{\iota \FM}(\iota Q)}^{\iota Q\,C_G(\iota Q)}(k)$
is indecomposable by Green's Indecomposability Theorem.

If $Q \leq O_2(\CF)$, we have two cases: either $C_{\FM}(Q)$ is a group of composite order
(that is $|C_{\FM}(Q)|=3.2^k$ for $k\leq n$) or $C_{\FM}(Q)$ is a $2$-group.
In the first case, since $C_{\iota \FM}(\iota Q)\le C_G(\iota Q)$, $3$ divides
both of the orders $C_G(\iota Q)$ and $C_{\iota \FM}(\iota Q)$ so that the index $|C_G(\iota Q):C_{\iota \FM}(\iota Q)|$
is a power of $2$, hence the corresponding induced module is indecomposable by
Lemma \ref{Green}. If $C_{\FM}(Q)$ is a $2$-group, then $C_{\FM}(Q)$ lies
inside one of the conjugates of $P$, since $P$ is a Sylow $p$-subgroup of $\FM$.
Suppose that $C_{\FM}(Q) \leq P$. Then since $Q\leq O_2(\CF)$, by Lemma \ref{Op(F)}
$$\iota Q = \{(u^{m_1}, u^{m_2}, u^{m_3}; \id) \ | \ u \in Q\}$$
where $\{m_1, m_2, m_3\}$ is a set of coset representatives of $P$
in $\FM$. So $(x_1, x_2, x_3; \beta) \in C_G(\iota Q)$ if and only if
$ m_{\beta^{-1}(i)} x_i^{-1} m_i^{-1} \in C_{\FM}(Q)$.
Without loss of generality $m_1$ can be chosen as the representative of
the coset $P$. Putting $i=1$ in the last condition, we get $m_{\beta^{-1}(1)} P=P$
since $C_{\FM}(Q)\leq P$. We get $\beta (1)=1$ and so $\beta$ is either identity
or a $2$-cycle. Therefore $C_G(\iota Q)$
is a $2$-group and the indecomposability of $\Ind_{\iota Q\,C_{\iota \FM}(\iota Q)}^{\iota Q\,C_G(\iota Q)}(k)$
follows from Green's Indecomposability Theorem. If $C_{\FM}(Q) \not \leq P$, there
exists an element $m\in \FM$ such that $C_{\FM}(Q) \leq P^m.$ From this we get $C_{\FM}( ^m Q) \leq P.$ Thus
$C_G(\iota (^m Q))$ is a $2$-group by the argument above. Since $C_G(\iota ( ^m Q)) \cong C_G(\iota Q)$, the centralizer
$C_G(\iota Q)$ is a $2$-group. Hence indecomposability of
$\Ind_{\iota Q\,C_{\iota \FM}(\iota Q)}^{\iota Q\,C_G(\iota Q)}(k)$ follows from
Green's Indecomposability Theorem. Therefore, the Brauer indecomposability is established.

\end{proof}

\section*{Acknowledgements}

\noindent
{\small
The authors would like to thank the referees for their careful reading 
of the first manuscript and 
for valuable comments.
A part of this work was done while the second author was visiting
Chiba University in July 2017. She would like to thank
the Center for Frontier Science, Chiba University for their hospitality.
She would like to thank also Naoko Kunugi for her hospitality.}

\section*{Funding}

\noindent
{\small
The first author was supported in part by  the Japan Society for
Promotion of Science (JSPS), Grant-in-Aid for Scientific Research
(C)15K04776, 2015--2018.
The second author was partially supported by the Center for Frontier Science,
Chiba University and Mimar Sinan Fine Arts University Scientific Research Project Unit
with project number 2017/22.}

\end{document}